\documentclass[10pt,a4paper]{article}
\usepackage[utf8]{inputenc}
\usepackage[english]{babel}
\usepackage{amsmath}
\usepackage{amsfonts}
\usepackage{amsthm}
\usepackage{amssymb}
\usepackage{color}

\usepackage[margin=2cm]{geometry}

\usepackage[colorlinks=false]{hyperref}

\newtheorem{theorem}{Theorem}[section]

\newtheorem{remark}{Remark}[section]

\numberwithin{equation}{section}

\allowdisplaybreaks

\newcommand{\R}{{\mathbb R}}

\newcommand{\N}{{\mathbb N}}

\newcommand{\D}{{\mathbb D}}
\def\D{\mathbb D}

\begin{document}

\title{On the Laplace equation with non-local dynamical boundary conditions}

\author{\textsc{Raffaela Capitanelli \& Mirko D'Ovidio}\\
raffaela.capitanelli@uniroma1.it, mirko.dovidio@uniroma1.it\\
 \textit{Sapienza University of Rome, Rome, Italy}}
\maketitle

\begin{abstract}
Aim of the paper is to study non-local dynamic boundary conditions of reactive-diffusive type for the Laplace equation from analytic and probabilistic point of view. In particular, we  provide  compact and probabilistic representation of the solution together with an interpretation in term of boundary processes.
\end{abstract}




\section{Introduction}

Aim of the paper is to study  non-local dynamic boundary conditions of reactive-diffusive type  for  the Laplace equation. More precisely, for $k \in \mathbb{R}$ and $l >0$, we consider the following problem

\begin{align}
\label{E}
\left\lbrace
\begin{array}{ll}
\displaystyle \Delta u=0 \qquad
&\text{in
$(0,\infty)\times\Omega$,}\\
 D_t^{\alpha} u = k \partial_{\bf n} u +l\Delta_\Gamma u \qquad &\text{on
$(0,\infty)\times \Gamma$,}\\
u(0,x)=u_0(x) &
 \text{on $ \Gamma$,}
\end{array}
\right .
\end{align}
where $u=u(t,x)$, $t\in  [0, \infty)$, $x\in\Omega$ and
$\Gamma=\partial\Omega$.
The symbol $\Delta_\Gamma$ denotes the Laplace--Beltrami operator on $\Gamma$ and $\partial_{\bf n} u$ is the normal derivative being ${\bf n}$ the outward normal to $\Omega.$  We assume that $\Omega$ is a $C^\infty$ bounded domain of $\R^N$ ($N\ge 2$) and $\Gamma$ is a Riemannian manifold endowed with the natural  inherited metric. The non-local operator 
\begin{align}
\label{dc}
D^\alpha_t  u(t) = \frac{1}{\Gamma(1-\alpha)} \int_0^t u^\prime(s) (t-s)^{-\alpha} ds, \quad \alpha \in (0,1)
\end{align}
is the Caputo-D\v{z}rba\v{s}jan derivative. We recall that, for $z>0$, $\Gamma(z) = \int_0^\infty s^{z-1} e^{-s}ds$ is the absolutely convergent Euler integral also termed gamma function.\\

We provide a compact representation for the solution to the problem \eqref{E}. We exploit the result in \cite{VV1} for which adding the Laplace--Beltrami operator on the boundary ($l > 0$) the corresponding problem is well-posed for any value of $k$. Moreover, under suitable spectral conditions, we show that such solution has the probabilistic representation
\begin{align}
\mathbf{E}_{\mu(x)}[u_0(X^\Gamma \circ L_t)] = \int_\Gamma \mathbf{E}_y[u_0(X^\Gamma \circ L_t)] \mu(x,dy), \quad x \in \Omega \cup \Gamma
\label{repIniDist}
\end{align}
where the initial distribution $\mu(x,dy)$ is an harmonic measure and $X^\Gamma$ is a boundary Markov process with random time $L$ explaining the non-local dynamic. Thus, our problem can be treated in terms of a fractional Cauchy problem on compact manifold without boundary for which 
\begin{align*}
\mathbf{E}_{\mu(x)}[u_0(X^\Gamma \circ L_t)] = \mathbf{E}_x[u_0(X^\Gamma \circ L_t)], \quad x \in \Gamma
\end{align*}
gives the solution on $\Gamma$ according with the Poisson integral \eqref{repIniDist}.\\

The operator \eqref{dc} has been introduced in \cite{caputoBook,CapMai71,CapMai71b} by the first author and separately in a series of works starting from \cite{Dzh66,DzhNers68} by the second author.  For $\alpha=1$ the operator $D^\alpha_t$ becomes the ordinary derivative for which we get the dynamical boundary condition of reactive-diffusive type
\begin{align}
\dot{u} = k \partial_{\bf n} u + l \Delta_\Gamma u \quad \textrm{on } (0,\infty) \times \Gamma.
\label{RDtypeC}
\end{align}
From an analytical point of view, such  a  \textit{non fractional} dynamical boundary condition has been studied in \cite{VV1} for the Laplace equation.
In particular,  such a condition describes a heat conduction process in $\Omega$ with a heat source on the boundary which can depend on the heat flux around the boundary and on the heat flux across it (see for example \cite{GM} and the reference therein). We remark that in the case $l=0$, the problem has been studied by many authors (see for example \cite{E},  \cite{H}, \cite{JL}), in this case the corresponding problem is well-posed if and only if $k\leq 0$. For the parabolic case associated with \eqref{RDtypeC} we mainly refer to \cite{VV2} where the authors considered the heat equation in place of the Laplace equation. We refer to \cite{Dov24II} for the parabolic case associated with the non-local condition in \eqref{E} which has been investigated only in case $l=0$ and $k<0$. For the fractional Cauchy problem we may refer to a very large literature. Such problems are well-know and they have been investigated by many researchers on both regular domains (see for example \cite{Chen17, GLY15, Koc2011, MLP2001, MNV09, Toaldo2015}) or irregular domains (see for example \cite{CCL, CapDovFEforms19, CapDovFCP21}). Recently the non-local dynamic boundary conditions for the heat equation have been investigated by \cite{Dov22fcaa, Dov24fcaa} together with the associated probabilistic representation in terms of sticky Brownian motions.

The plan of the paper is the following.
In the second section  we study the  Laplace equation with fractional dynamical boundary conditions
 from the  analytical point of view and in particular we  obtain a representation formula for the solution.
 In the third  section  we  deal with   the  same problem   from the probabilistic  point of view  and in particular we  give a  interpretation for the solution in term of boundary processes.
 In the last section we  consider  the  Laplace equation with a general non local boundary condition.

\section{Compact representation}

In this section we study the Laplace equation   on a bounded regular
domain $\Omega$ of $\R^N$ ($N\ge 2$)   with  time-fractional dynamical
boundary condition of reactive--diffusive type  described in \eqref{E}.

We denote by $\mathcal{D}(\Gamma )$ the space  $C^\infty(\Gamma )$  and  by $\mathcal{D}'(\Gamma)$  its dual space, that is the space of distributions on $\Gamma$. Moreover for any $s \in \R$  we denote by  the symbols $H^s(\Omega)$ and $H^s(\Gamma)$  the Sobolev spaces of  distributions
respectively on $\Omega$  and on its boundary  $ \Gamma$ (see  \cite{LM}).

We recall that 
the Laplace-Beltrami operator $\Delta_\Gamma$ can be  defined on $C^\infty(\Gamma)$
by the formula
\begin{equation}\label{M1}
-\int_\Gamma (\Delta_\Gamma u)  v \,ds=\int _\Gamma \nabla_\Gamma u\,  \nabla_\Gamma v \,ds
\end{equation}
for any $u,v\in C^\infty(\Gamma)$ where 
we denote by $\nabla_\Gamma$ the Riemannian gradient and by $ds$ the natural volume element on $\Gamma.$
We recall that, for any $s\in \R $, the operator
$ \Delta_\Gamma : D(\Delta_\Gamma) = H^{s+2}(\Gamma)\to H^s(\Gamma)$
 generates an analytic semigroup on $H^{s+2}(\Gamma)$  (see, for example,  Appendix A in   \cite{VV1}).

Before stating our result,  we recall  some properties of the  Dirichlet--to--Neumann operator $A_{DN}$ (see \cite{LM}).

For any $g\in H^s(\Gamma)$, $s>0$, the
non--homogeneous Dirichlet problem
\begin{equation}\label{NH}
  \begin{cases}
        \Delta v=0,& \text{in $\Omega$}, \\
        v =g & \text{on $\Gamma$,}
  \end{cases}
\end{equation}
has a unique solution $v\in H^{s+1/2}(\Omega).$
Moreover,  for any datum  $g\in H^s(\Gamma)$  of the non--homogeneous Dirichlet problem
\eqref{NH},    the operator   $\D $ that associates the datum  $g$ to the solution  $v= \D g $  is linear and bounded from $H^s(\Gamma)$  to $H^{s+\frac12}(\Omega).$
As $v$  has
normal derivative $\partial_{\bf n} v \in H^{s-1}(\Gamma),$  the Dirichlet--to--Neumann operator   $A_{DN}:$   $g \mapsto \partial_{\bf n} v$ is
bounded from $H^s(\Gamma)$ to $H^{s-1}(\Gamma).$

We point out that, for all $u,v\in
C^\infty(\Gamma)$,  by integrating by parts, we obtain
\begin{equation}\label{M2}
\int_\Gamma A_{DN} u\,  v ds=\int_\Omega\nabla (\D u)\nabla (\D
v) d \mu \end{equation}
(we denote by $d\mu$ the natural volume element on $\Omega)$.

 We have that $A_{DN}$   generates an analytic  semigroup. Furthermore it is possible to prove that for $l>0$
the operator $k A_{DN} + l \Delta_\Gamma $ with domain $D(k A_{DN} + l \Delta_\Gamma) = H^{s+2}(\Gamma)$  generates an analytic  semigroup in $H^s$ (see  the proof of Theorem 1 in  \cite{VV1}).

\bigskip

Now our aim is to prove a representation formula for the solution of problem  \eqref{E}  following  the same approach of \cite{VV1}.
We start by recalling  Theorem 2 from \cite{VV1}.

\begin{theorem}\label{th1}

Let $l > 0.$ Then there is an Hilbert basis $\{\varphi_n, n\in  \N\}$ of $L^2(\Gamma)$, $\varphi_n\in \mathcal{D}(\Gamma),$  $\varphi_n$ real valued, and an increasing real sequence $\{\lambda_n,
n\in  \N\}$,   $\lambda_n\to\infty $  as $ n\to\infty,$ each $\lambda_n$ having finite multiplicity, such that for any $n\in \N$ a unique solution  $\psi_n$  (which belongs to $C^\infty(\bar\Omega )$
of the Dirichlet non-homogeneous problem

\begin{align}
\label{E1}
\left\lbrace
\begin{array}{ll}
\displaystyle \Delta \psi_n=0 \qquad
&\text{in
$\Omega$,}\\
\psi_n(x)=\varphi_n(x)&\text{on $\Gamma$}
\end{array}
\right.
\end{align}
solves the eigenvalue problem
\begin{align}
\label{E2}
\left\lbrace
\begin{array}{ll}
\displaystyle \Delta \psi_n=0 \qquad
&\text{in
$\Omega$,}\\
\lambda_n\psi_n =-k \partial_{\bf n} \psi_n - l \Delta_\Gamma \psi_n \qquad &\text{on
$ \Gamma$.}
\end{array}
\right .
\end{align}
\end{theorem}

We now consider, for a real parameter $\Lambda,$ the operator 
$$A_{\Lambda} := - k A_{DN} - l \Delta_\Gamma + \Lambda I$$ 
in $H^s(\Gamma)$. It is possible to prove that there exists $\bar\Lambda\geq 0$ such that for any  $\Lambda\geq \bar\Lambda$ the following problem $$- k A_{DN} v- l \Delta_\Gamma v + \Lambda v=h$$ with $h$ an arbitrary element of $H^s(\Gamma)$ has a unique solution $v\in H^{s+2}(\Gamma)$ (see Lemma 1 in  \cite{VV1}).
Then from now on  we put $\Lambda=\bar\Lambda$ without loss of generality.
Denote by $ <\cdot ,\cdot >$ the norm $ (\cdot ,\cdot )_{L^2(\Gamma)}$.

For the operator $$A_{\Lambda}:D(A_{\Lambda})=H^2(\Gamma)\subset L^2(\Gamma)\to L^2(\Gamma)$$
we consider  nonnegative real powers $$A^{s/2}_{\Lambda}:D(A^{s/2}_{\Lambda})\subset L^2(\Gamma)\to L^2(\Gamma)$$
where $$D(A^{s/2}_{\Lambda})=\{ u\in L^2(\Gamma) : \sum_{n=1}^\infty | <u, \varphi_n>|^2 (\lambda_n+\Lambda)^s<\infty)\}$$
and $$A^{s/2}_{\Lambda}=\sum_{n=1}^\infty   <u, \varphi_n> (\lambda_n+\Lambda)^{\frac{s}2} \varphi_n$$ for $u\in D(A^{s/2}_{\Lambda}).$

It is possible to prove that $H^s(\Gamma)=D(A^{s/2}_{\Lambda})$ (see the proof of Theorem 2 in \cite{VV1}).
Then we equip $H^s(\Gamma),$ $s\in \R,$ with the scalar product $$((u,v))_{H^s(\Gamma)}=(A_\Lambda^{s/2} u, A_\Lambda^{s/2} v)_{L^2(\Gamma)}$$
and norm $$||| \cdot|||_{H^s(\Gamma)} =||A_\Lambda^{s/2} u||_{L^2(\Gamma)}$$ which is equivalent to the standard one.
Then  $ \{\varphi_n/|||\varphi_n|||_{H^s(\Gamma)}, n \in \N\}$ is an Hilbert basis of $H^s(\Gamma)$ for all $s\in \R,$ provided this space is endowed with  the  scalar
product $((\cdot ,\cdot ))_{H^s(\Gamma)} $ and norm $|||\cdot|||_{H^s( \Gamma)}$
We use the following characterization of Sobolev space $$H^s(\Gamma)=\{  u\in \mathcal{D}':   \sum_{n=1}^\infty | <u, \varphi_n>|^2 (\lambda_n+\Lambda)^s <\infty \}$$ (for the proof, see, for example,  Appendix B in \cite{VV1}).

Further on we consider the problem
\begin{align}
\label{ELambda}
\left\lbrace
\begin{array}{ll}
\displaystyle \Delta u=0 \qquad
&\text{in
$(0,\infty)\times\Omega$,}\\
 D_t^{\alpha} u = -A_\Lambda u \qquad &\text{on $(0,\infty)\times \Gamma$,}\\
u(0,x)=u_0(x) &
 \text{on $ \Gamma$,}
\end{array}
\right .
\end{align}
which coincides with \eqref{E} for $\Lambda=0$.  We say  that  $u$  is  solution  of problem  \eqref{ELambda}  if 
$$u= \D v$$
$$v\in C([0,\infty);H^{s}(\Gamma))$$ 
$$ D^\alpha_t  v \in C((0,\infty);H^{s}(\Gamma)) $$
such that 
\begin{align*}
D_t^{\alpha} v = - A_\Lambda v \;  \textrm{ on $H^s(\Gamma)$ for all $t>0$ with $v(0,x)=u_0(x)$  on $\Gamma$.}  
\end{align*}
We point out that as $u= \D v$ then  $u\in C^{\infty}(\Gamma)$ and $ \Delta u=0$  in classical sense.\\

We show below that the solution to $D_t^{\alpha} v = k \partial_{\bf n} v +l\Delta_\Gamma v - \Lambda v$ on $H^s(\Gamma)$ with initial datum $u_0 \in H^s(\Gamma)$ can be obtained, under suitable conditions, via time change according with the well-established theory associated with fractional Cauchy problems on compact manifold without boundary. \\

Thus, we provide spectral conditions for which, given any $u_0\in H^s(\Gamma)$ there is a unique  solution $v$ satisfying the time fractional equation  $D_t^{\alpha} v = k \partial_{\bf n} v +l\Delta_\Gamma v -\Lambda v$ on the boundary   $\Gamma$  with the initial condition $v(0,x)=u_0(x)$ on $ \Gamma$. Existence and uniqueness of the solution to problem  \eqref{E} is therefore obtained by noticing that $u = \D v$.

By using the previous arguments, we  prove the following  representation formula (see also Remark \ref{rmk:IMP}). 

First we recall that $E_{\alpha, \beta}$ with $\alpha,\beta>0$ is the well-known Mittag-Leffler function (see \cite{GorKilMaiRog} for definition and properties).\\

\begin{theorem}
\label{th2}
 Let $\Lambda=0$, $l > 0$, $k \in \mathbb{R}$ and $\alpha \in (0,1]$.  Let assume  $\lambda_1\geq 0.$
  Then for any $u_0 \in H^s( \Gamma)$ the solution $u$ of problem \eqref{ELambda} is given by
\begin{align}
u(t, x) = \sum_{n=1}^\infty
<u_0,\varphi_n>E_{\alpha,1}(-t^{\alpha}\lambda_n)\psi_n(x)
\label{uAlfa}
\end{align}  
the series being convergent in $C((0,\infty) \times \bar\Omega).$
\end{theorem}

\begin{proof}  
Let us consider $\alpha \in (0,1)$. By using the method of separation of variables, we put
 $u(x, t) = G(t)F(x).$   Then substituting into the boundary condition in \eqref{E}, we obtain on $\Gamma$
$$ F(x) D_t^{\alpha} G(t)= G(t)( k  \partial_{\bf n} F(x) +l\Delta_\Gamma F(x)) $$  and therefore

$$ \frac{ D_t^{\alpha} G(t)}{G(t)}= \frac{ k \partial_{\bf n} F(x) +l\Delta_\Gamma F(x)}{F(x)}=-\lambda$$ for some $\lambda>0.$

Then we obtain  
\begin{equation}
\label{1} 
k \partial_{\bf n} F(x) + l\Delta_\Gamma F(x)=-\lambda F(x) \end{equation}
and  \begin{equation}\label{2}D_t^{\alpha} G(t)= -\lambda G(t).\end{equation}

From Theorem \ref{th1},   we have that  there is an Hilbert basis $\{\varphi_n, n\in  \N\}$ of $L^2(\Gamma)$, $\varphi_n\in \mathcal{D}(\Gamma),$  $\varphi_n$ real valued, and an increasing real sequence $\{\lambda_n,
n\in  \N\}$,   $\lambda_n\to\infty $  as $ n\to\infty,$ each $\lambda_n$ having finite multiplicity, such that for any $n\in \N$
  the first eigenvalue problem \eqref{1} is solved by   
  $$k \partial_{\bf n} \psi_n   + l \Delta_\Gamma \psi_n = - \lambda_n\psi_n .$$

It is well-known that the solution to
\begin{align}
D^\alpha_t \varphi = - c\, \varphi, \quad \varphi(0)=1, \quad c>0,
\label{eqML}
\end{align}
is given by the Mittag-Leffler function $\varphi(t)=E_{\alpha,1}(-c t^\alpha)$. 
Then  second eigenvalue problem \eqref{2} with $\lambda=\lambda_n$ is solved by the function  $ E_{\alpha,1}(-t^{\alpha}\lambda_n).$

By considering  the  initial datum $$u_0=\sum_{n=1}^\infty
<u_0,\varphi_n>\varphi_n$$
in $H^s( \Gamma)$
we obtain  that for $x\in \Gamma$
 $$u(t, x) =\sum_{n=1}^\infty
<u_0,\varphi_n>E_{\alpha,1}(-t^{\alpha}\lambda_n)\varphi_n(x)$$ in $H^s( \Gamma)$  for all $t\geq 0$ where the convergence is  in the  $C([0,\infty); H^s(\Gamma )).$

In fact we have that  the serie converges in  $C([0,T); H^s(\Gamma ))$  for all $T > 0$ as
$$|||\varphi_n|||^2_{H^s(\Gamma)}=(\lambda_n+\Lambda)^s$$ for all $s\in   \R$ and $n\in \N$ and so
$$||| <u_0,\varphi_n>E_{\alpha,1}(-t^{\alpha}\lambda_n)\varphi_n(x)|||^2_{H^s(\Gamma)}\leq | <u_0,\varphi_n> |^2 (E_{\alpha,1}(-t^{\alpha}\lambda_1))^2 (\lambda_n+\Lambda)^s$$
by using the fact that $E_{\alpha,1}(-t^{\alpha}\lambda_n)$ is completely monotone for $0\leq t \leq T$ and  $0 < \alpha < 1$ and $\lambda_n\geq \lambda_1.$ 
Therefore as $u_0\in H^s(\Gamma )$ and then $\sum_{n=1}^{\infty} | <u_0,\varphi_n> |^2(\lambda+\Lambda)^s<\infty$  we obtain that  
 $$\sum_{n=1}^{\infty} || <u_0,\varphi_n>E_{\alpha,1}(-t^{\alpha}\lambda_n)\varphi_n(x)||^2_{C([0,T], H^s(\Gamma))}< \infty$$
and then the serie converges in  $C([0,T); H^s(\Gamma ))$ for all $T > 0$  and therefore in  $C([0,\infty); H^s(\Gamma)).$

Because, for any datum  $g\in H^s(\Gamma)$  of the non--homogeneous Dirichlet problem
\eqref{NH},    the operator  $ \D$ that associates the datum  $g$ to the solution  $v= \D g$  is linear and bounded from $H^s(\Gamma)$  to $H^{s+\frac12}(\Omega),$  by considering  $\D\varphi_n=\psi_n,$ we obtain that
$$u(t, x) =\sum_{n=1}^\infty
<u_0,\varphi_n>E_{\alpha,1}(-t^{\alpha}\lambda_n)\psi_n(x)$$  converges in $H^{s+1/2}( \Omega)$  for all $t\geq 0$ where the convergence is  in the  $C([0,\infty); H^{s+1/2}(\Omega )) $ topology.

Now we consider  the time fractional derivative of $$u(t, x) =\sum_{n=1}^\infty
<u_0,\varphi_n>E_{\alpha,1}(-t^{\alpha}\lambda_n)\varphi_n(x).
$$ in $C([\varepsilon,T); H^s(\Gamma))$ for all $0<\varepsilon<T<\infty.$

We prove the convergence of  $$\sum_{n=1}^\infty
<u_0,\varphi_n> (-\lambda_n)E_{\alpha,1}(-t^{\alpha}\lambda_n)\varphi_n(x)
$$
by  considering $$||| <u_0,\varphi_n> (-\lambda_n) E_{\alpha,1}(-t^{\alpha}\lambda_n)\varphi_n(x)|||^2_{H^s(\Gamma)}\leq | <u_0,\varphi_n> |^2 (E_{\alpha,1}(-t^{\alpha}\lambda_n))^2 |\lambda_n|^2 (\lambda_n+\Lambda)^s.$$

By using again the  property of  $E_{\alpha,1}$  with $0 < \alpha < 1,$   (see Theorem 1.6 in \cite{POD})  we have that   $$|(E_{\alpha,1}(-t^{\alpha}\lambda_n))^2 |\lambda_n|^2\leq (\frac 1{1+ t^{\alpha}\lambda_n})^2 |\lambda_n|^2.$$ 

Then there exist $n_0$ such that, for any $n>n_0$
  for  $0<\varepsilon<t,$ we have  $$||| <u_0,\varphi_n> (-\lambda_n) E_{\alpha,1}(-t^{\alpha}\lambda_n)\varphi_n(x)|||^2_{H^s(\Gamma)}\leq  C_\varepsilon | <u_0,\varphi_n> |^2  (\lambda_n+\Lambda)^s$$ and then the serie converges in $C([\varepsilon,T); H^s(\Gamma))$ for all $0<\varepsilon<T<\infty.$

The Laplace equation  on the bulk is trivially satisfied for the construction of $ \psi_n.$

For $\alpha=1$, it is well-known that $E_{1,1}(z) = e^{-z}$ and then the solution $u$ takes the form
\begin{align}
u(t, x) =\sum_{n=1}^\infty
<u_0,\varphi_n> e^{-t \lambda_n}\psi_n(x)
\label{uUno}
\end{align}

which is the solution as it has been proved in Theorem 3 of  \cite{VV1}.

\end{proof}

\begin{remark}
Observe that the spectral condition is not  trivial. We point out that, for example, the  condition  $\lambda_1\geq 0$ is satisfied when  $\Omega$ is an  Euclidean ball with radius $R$  in $\R^N$ if $k \leq l R (N-1)$
(see Theorem 6 in \cite{VV1}).
\label{rmk:IMP}
\end{remark}

Recall that 
\begin{align*}
-A_\Lambda = k A_{DN} + l \Delta_\Gamma - \Lambda I, \quad k \in \mathbb{R},\, l> 0,\, \Lambda > 0
\end{align*}
is self-adjoint, positive operator (see \cite[proof of Theorem 2]{VV1}). By using the  one to one correspondence between the non-positive self-adjoint operator $-A_\Lambda$ and the family of closed symmetric form (see Theorem 1.3.1 in \cite{FUK}) we consider the closed symmetric form   $\mathcal{E} $
defined as  
\begin{align}
\mathcal{E}(u,v)=(\sqrt{A_\Lambda u}, \sqrt{A_\Lambda v})\end{align}
with domain  $D(\mathcal{E}) =D(\sqrt{A_\Lambda u})$. In particular, by using \eqref{M1} and \eqref{M2} we obtain
\begin{align}
\mathcal{E}(u,v) =  - k \int_\Omega\nabla (\D u)\nabla (\D
v) \,d\mu+ l \int _\Gamma \nabla_\Gamma u\,  \nabla_\Gamma v \,ds + \Lambda \int _\Gamma uv\,ds
\label{DFprocess}.
\end{align}
For $k \leq 0$ we get a Dirichlet form 
 with Markovian semigroups and resolvents (see, for definitions and properties, see \cite{FUK}).

We remark that the previous form can be associated to the so-called Venttsel' or  Wentzell  boundary condition (see for example  pag 26 in \cite{{GWbook}}).
We point out that recently  fractional-in-time Venttsel' problems in irregular  domains  has been studied  in \cite{CCL}  by proving  well-posedness,  regularity and asymptotic results. \\

In the next section we consider the probabilistic representation of the solution.

\section{Probabilistic representation}

We introduce briefly the processes we deal with further on. As usual, we denote by $\mathbf{E}_x$ the mean operator under the measure $\mathbf{P}_x$ where $x$ is a starting point:
\begin{itemize}
\item[i)] $X^+= \{X^+_t\}_{t\geq 0}$ is a reflecting Brownian motion on $\overline{\Omega}$ with boundary local time $\gamma^+ = \{\gamma^+_t\}_{t\geq 0}$. The process $X^+$ has generator $(G^+, D(G^+))$ where $G^+=\Delta$ is the Neumann Laplacian with 
\begin{align*}
D(G^+)= \{\varphi, \Delta \varphi \in C(\overline{\Omega}),\, \varphi \in H^1(\Omega)\,:\, \partial_{\bf n} \varphi = 0\};
\end{align*}
\item[ii)] $H=\{H_t\}_{t\geq 0}$ is a subordinator with $\mathbf{E}_0[\exp(-\lambda H_t)] = \exp(-t \Phi(\lambda))$, $\lambda  >0$;
\item[iii)] $L=\{L_t\}_{t\geq 0}$ is the inverse $L_t= \inf\{s \geq 0\,:\, H_s >t\}$ to $H$.
\end{itemize}
Moreover, we introduce the first hitting time
\begin{align*}
\tau := \inf\{t\,:\, X^+_t \in \Gamma\}.
\end{align*}
Notice that we may equivalently consider the first hitting time of a Brownian motion on $\mathbb{R}^d$ started at $x \in \Omega$. We  underline that $X^+$ will be a reference (base) process to start with all over the paper. The symbol \begin{align}
\label{symbPhi}
\Phi(\lambda) = \int_0^\infty \left( 1 - e^{-\lambda t} \right)\Pi(dt), \quad \lambda > 0
\end{align}
is the Bernstein function (or L\'{e}vy symbol) characterizing the subordinator $H$.   We also introduce the Caputo-D\v{z}rba\v{s}jan type derivative defined as the convolution operator
\begin{align*}
D^\Phi_t \varphi(t) = (\varphi^\prime * \kappa)(t), \quad \alpha \in (0,1)
\end{align*}
involving the singular kernel $\kappa$ for which
\begin{align*}
\int_0^\infty e^{-\lambda t} \kappa(t)dt = \frac{\Phi(\lambda)}{\lambda}, \quad \lambda>0.
\end{align*} 
It is well-known that $\kappa$ can be identified as the tail of a L\'{e}vy measure $\Pi$ in \eqref{symbPhi}. In particular, $\kappa(t) = \Pi(t, \infty)$. 

We point out that for $\Phi(\lambda)=\lambda^\alpha$ with $\alpha \in (0,1)$, that is the case of stable subordinators, $D^\Phi_t = D_t^{\alpha}$ introduced in \eqref{dc}.
For a discussion on $H$ and $L$ we refer to the book \cite{Ber99}. Here we only provide the result in Theorem \ref{thm:eqM} below for the sake of completeness.

Our discussion in this section will focus on the following facts. The well-known Dynkin's formula 
\begin{align*}
\mathbf{E}_x[u(X^+_\tau)] = u(x)  +  \mathbf{E}_x[\int_0^\tau \Delta u(X^+_s) ds]
\end{align*}
holds in case $\Omega$ is a transient domain (that is $\mathbf{P}_x(\tau < \infty) = 1$ for every $x \in \Omega$) with $\mathbf{E}_x[\tau] < \infty$ for every $x \in \Omega$. Moreover, we ask for $u$ to be continuous and bounded on $\overline{\Omega}$.  

If $u$ is harmonic, that is $\Delta u=0$ in $\Omega$, then (also neglecting the condition $\mathbf{E}_x[\tau] < \infty$) the Dynkin's formula writes $\mathbf{E}_x[u(X^+_\tau)] = u(x)$ where $\mathbf{P}_x(X^+_\tau \in dy)$ plays the role of harmonic measure (or Poisson kernel). In particular, $u$ must satisfy the mean value property in $\Omega$ and this implies that $u \in C^\infty(\overline{\Omega})$. 

If $\Omega$ is a regular (that is $\mathbf{P}_x(\tau =0)=1$ for every $x \in \Gamma$) and  transient domain, then every bounded and continuous function on $\Gamma$ has a unique, bounded, harmonic extension.

We recall that every bounded domain is transient. In our case, in which $\Gamma$ is a smooth boundary, the harmonic measure is absolutely continuous with respect to the surface measure and the Poisson kernel can be explicitly calculated depending on the dimension $N$. It worth noticing what follows. For $N\geq 3$ the Brownian motion is transient. For $N=2$ the Brownian motion visits a given point with probability zero whereas it visits any neighbourhood of that point with probability one. For $N=1$ the Brownian motion visits a point infinitely many times and the corresponding set of hitting times is a perfect set (a large and uncountable set of zero Lebesgue measure, there are no isolated points). We only consider $\Omega \subset \mathbb{R}^N$ with $N\geq 2$.

\begin{theorem}
Let us write $M_\Phi(t, \theta) = \mathbf{E}_0[\exp(-\theta L_t)]$, $t>0$, $\theta>0$. Then, $\varphi(t) = M^0_\Phi(t, \theta)$ is the unique solution to
\begin{equation*}
\left\lbrace
\begin{array}{ll}
\displaystyle D^\Phi_t \, \varphi(t) = - \theta\, \varphi(t), & t>0,\\
\displaystyle \varphi(0)=1.
\end{array}
\right. 
\end{equation*} 
\label{thm:eqM}
\end{theorem}
\begin{proof}
The proof follows by standard arguments. Since $D^\Phi_t$ is a convolution-type operator, we get
\begin{align*}
\int_0^\infty e^{-\lambda t} D^\Phi_t \varphi\, dt 
= & \left( \int_0^\infty e^{-\lambda t} \kappa(t)dt \right) \left( \int_0^\infty e^{-\lambda t} \varphi^\prime(t)\, dt \right)\\
= & \frac{\Phi(\lambda)}{\lambda}(\lambda \tilde{\varphi}(\lambda) - 1), \quad \lambda>0.
\end{align*}
Our problem leads to
\begin{align*}
\Phi(\lambda)\, \tilde{\varphi}(\lambda) + \theta \tilde{\varphi}(\lambda) = \frac{\Phi(\lambda)}{\lambda} \quad \textrm{from which} \quad \tilde{\varphi}(\lambda) = \frac{\Phi(\lambda)}{\lambda} \frac{1}{\theta + \Phi(\lambda)}.
\end{align*}
We immediately see that
\begin{align*}
\tilde{\varphi}(\lambda) = \int_0^\infty e^{-\theta s} \frac{\Phi(\lambda)}{\lambda} e^{-s \Phi(\lambda)} ds =  \int_0^\infty e^{-\lambda t} \int_0^\infty e^{-\theta s} \mathbf{P}_0(L_t \in ds) dt
\end{align*}
where the last identity follows from $\mathbf{P}_0(L_t < s) = \mathbf{P}_0(H_s > t)$. Thus, $\varphi(t) = \mathbf{E}_0 [\exp(-\theta L_t)]$ which is continuous on $[0, \infty)$ and this implies uniqueness of the inverse Laplace transform.
\end{proof}

Further on we consider $\Phi(\lambda)=\lambda^\alpha$ with $\alpha \in (0,1).$ Only our last result is concerned with a general $\Phi$. Moreover, we underline that
\begin{align}
-A_\Lambda = kA_{DN} + l\Delta_\Gamma - \Lambda I \quad \textrm{in case} \quad k \leq 0, \; l \geq 0, \; \Lambda\geq 0
\label{Agen}
\end{align}
with
\begin{align*}
D(-A_\Lambda) \subset C(\Gamma) \cap D(A^{s/2}_\Lambda)
\end{align*}
plays a special role. In particular,  $(-A_0, D(-A_0))$ where
\begin{align*}
-A_0 = kA_{DN} + l\Delta_\Gamma \quad \textrm{with} \quad k<0,\; l \geq 0
\end{align*}
generates a Markov process on $\Gamma$ with right-continuous paths. The negative Dirichlet-Neumann operator $-A_{DN}$ is a non-local operator introducing jumps and the Laplace-Beltrami operator $\Delta_\Gamma$ gives the diffusive part. \\

Let us denote by $X^\Gamma=\{X^\Gamma_t\}_{t\geq 0}$ the process on $\Gamma$ with generator $(-A_\Lambda, D(-A_\Lambda))$. Since $\Omega$ is smooth, then there exist a unique linear and bounded trace operator $T:  H^1(\Omega) \to L^2(\Gamma)$ such that $Tu= u|_\Gamma$. Consider the Dirichlet form
\begin{align}
\mathcal{E}(Tu,Tv) := \int_\Omega \nabla u \nabla v\,  dx, \quad u,v \in F := \{\varphi \in H^1(\Omega)\,:\, \Delta \varphi = 0\}
\label{DFDN}
\end{align}
with $D(\mathcal{E}) := \{ T\varphi\,:\, \varphi \in H^1(\Omega) \}$. Thus, $D(\mathcal{E}) = H^{1/2}(\Gamma)$ is the trace space. There exists a Markov process on $\Gamma$ associated with $(\mathcal{E}, D(\mathcal{E}))$, see \cite[Proposition 3.1]{BS}. In particular, for $l=0$ and $\Lambda=0$, we say that $X^\Gamma$ is the process on $\Gamma$ associated with the Dirichlet form $(\mathcal{E}, D(\mathcal{E}))$ given in \eqref{DFDN} for $k<0$. It is therefore a pure jump process with generator $(-A_{DN}, D(-A_{DN}))$ up to some scaling introduced by $|k|$. The operator $\Delta_\Gamma$ is associated with a Brownian diffusion on the compact manifold $\Gamma$. It is well-known that we have a Markov semigroup with compact representation. The conservative part introduced by $\Lambda>0$ can be treated in a standard way. For $k, l, \Lambda$ as specified in \eqref{Agen}, the Dirichlet form \eqref{DFprocess} ensures existence of a stochastic process associated with $(-A_\Lambda, D(-A_\Lambda))$. Moreover, from the theory of Dirichlet form, we have a Markov semigroup.

We now underline a connection with time-changed semigroups. We recall that, for $k \in \mathbb{R}$ and $l>0$, the problem 
\begin{align}
\left\lbrace
\begin{array}{ll}
\displaystyle \Delta w =0 \qquad
&\text{in
$(0,\infty)\times\Omega$,}\\
\dot{w} = k \partial_{\bf n} w +l\Delta_\Gamma w \qquad &\text{on
$(0,\infty)\times \Gamma$,}\\
w(0,x)=u_0(x) &
 \text{on $ \Gamma$,}
\end{array}
\right .
\label{EsolVVw}
\end{align}
has been investigated in \cite{VV1} where they obtained an analytic semigroup with compact representation. We restrict our analysis in case of non negative eigenvalues of $-A_0$.

\begin{theorem}
Let $\Lambda=0$, $l > 0$, $k \in \mathbb{R}$ and $\alpha \in (0,1]$.  Let assume  $\lambda_1\geq 0.$ Let $w$ be the solution to \eqref{EsolVVw}. Then, the solution \eqref{uAlfa} can be written as 
\begin{align*}
u(t,x) = \int_0^\infty w(s,x) \mathbf{P}_0(L_t \in ds)
\end{align*}
where 
\begin{align}
\int_0^\infty e^{-\lambda t} \mathbf{P}_0(L_t \in ds) = E_\alpha(-\lambda t^\alpha), \quad t \geq 0,\; \lambda \geq 0.
\label{ContDensL}
\end{align}
\label{thm:SOLw}
\end{theorem}

\begin{proof}
It is well-known that \eqref{ContDensL} hold true (see \cite{Bingham71}). Indeed, $L$ is an inverse to a stable subordinator. Moreover,
\begin{align*}
E_\alpha(-\lambda t^\alpha) = \sum_{k \geq 0} \frac{(-\lambda t^\alpha)^k}{\Gamma(\alpha k + 1)} \quad \textrm{becomes} \quad E_1(-\lambda t) = e^{-\lambda t}
\end{align*}
for $\alpha=1$. From Theorem \ref{th2}, for $\alpha=1$, the solution $u$ takes the form 
\begin{align*}
u(t, x) =  \sum_{n=1}^\infty <u_0,\varphi_n> e^{-t \lambda_n}\psi_n(x)
\end{align*}
which can be associated with the fact that $L_t \to t$ a.s. as $\alpha \to 1$ (see \cite{Ber99}). Thus, we conclude that, $\forall\, T>0$, pointwise on $(0, T) \times \overline{\Omega}$, 
\begin{align*}
u = w \quad \textrm{for } \alpha=1
\end{align*}
Now we consider $\alpha \in (0,1)$.\\

The previous result says that
\begin{align*}
u(t, x) 
= & \sum_{n=1}^\infty <u_0,\varphi_n> E_{\alpha,1}(-t^{\alpha}\lambda_n)\psi_n(x)\\
= & \sum_{n=1}^\infty <u_0,\varphi_n> \left( \int_0^\infty e^{- s \lambda_n} \mathbf{P}_0(L_t \in ds) \right) \psi_n(x)\\
= & \int_0^\infty \left( \sum_{n=1}^\infty <u_0,\varphi_n>  e^{- s \lambda_n} \psi_n(x) \right) \mathbf{P}_0(L_t \in ds)\\
= & \int_0^\infty w(s,x)  \mathbf{P}_0(L_t \in ds)
\end{align*}
which is the claim.
\end{proof}

The last result says that the solution $u$ is written in terms of a time-changed process. 
This is not obvious and  will be associated below with  the fact that the behaviour of $X^\Gamma$ on the boundary is not affected from the behaviour on the interior. 
More precisely, we now provide the probabilistic interpretation  of the Laplace equation with fractional dynamical boundary condition.

We recall that, given a random time $T$ (or a stochastic process $T=\{T_t\}_{t\geq 0}$) and a process $Z = \{Z_t\}_{t\geq 0}$, we analogously write $Z_T$ or $Z \circ T$.

Moreover we  introduce the process $X^{0, \Gamma} = \{X^{0, \Gamma}_t\}_{t\geq 0}$ with generator $(-A_0, D(-A_0))$. In particular, we underline that $X^\Gamma = X^{0, \Gamma}$ in case $\Lambda=0$.

\begin{theorem}
\label{th3}
Let $\Lambda=0$, $l \geq 0$, $k <0$ and $\alpha \in (0,1]$. The solution to \eqref{ELambda} has the probabilistic representation
\begin{align}
u(t, x) =  \mathbf{E}_x[\mathbf{E}_{X^+_\tau}[u_0(X^{0, \Gamma}\circ L_t)]], \quad t>0, \; x \in \overline{\Omega}.
\label{ProbRepSol}
\end{align}
\end{theorem}

\begin{proof}
Let us consider the solution $u: (0, \infty)\times \overline{\Omega} \to \mathbb{R}$ to the problem
\begin{equation}
\left\lbrace
\begin{array}{ll}
\displaystyle \Delta u = 0, & \textrm{in } (0, \infty) \times \Omega,\\
\displaystyle Tu = v, & \textrm{on } (0, \infty) \times \Gamma.
\end{array}
\right .
\label{eqProofHarm}
\end{equation}
Then $\forall\, t \in (0, \infty)$, the solution $u(t, \cdot)$ has the probabilistic representation 
\begin{align}
u(t,x) = \mathbf{E}_x[v(t, X^+_{\tau})], \quad t\geq 0, \; x \in \overline{\Omega}.
\label{eq0}
\end{align}
Now, for $l=0$, we consider the datum $v: (0, \infty)\times \Gamma \to \mathbb{R}$ as the solution to the problem
\begin{equation}
\left\lbrace
\begin{array}{ll}
\displaystyle D^\alpha_t v = k A_{DN} v, & (0, \infty) \times \Gamma\\
\displaystyle v_0 = u_0 \in H^s(\Gamma)
\end{array}
\right.
\label{fcpDN}
\end{equation}
where $-A_{DN}$ is the generator of the right-continuous Markov process $X^{0, \Gamma}$ on $\Gamma$. Thus, for the fractional Cauchy problem \eqref{fcpDN} we have (see for example \cite{CapDovFEforms19})
\begin{align}
v(t,x) = \mathbf{E}_x[u_0(X^{0, \Gamma} \circ L_t)], \quad t\geq 0,\; x \in \Gamma.
\label{eq2}
\end{align}
Thus, 
\begin{align*}
u(t,x) \stackrel{\eqref{eq0}}{=} \mathbf{E}_x[v(t, X^+_\tau)] \stackrel{\eqref{eq2}}{=} \mathbf{E}_x\left[\mathbf{E}_{X^+_\tau}[u_0(X^{0, \Gamma}\circ L_t)]\right], \quad t\geq 0,\; x \in \overline{\Omega}.
\end{align*}
In case $l>0$, by considering the datum $v: (0, \infty)\times \Gamma \to \mathbb{R}$ as the solution to the problem
\begin{equation}
\left\lbrace
\begin{array}{ll}
\displaystyle D^\alpha_t v = k A_{DN} v + l\Delta_\Gamma v, \quad (0, \infty) \times \Gamma\\
\displaystyle v_0 = u_0 \in H^s(\Gamma)
\end{array}
\right.
\end{equation}
we include also the diffusive part, that is the Brownian motion on $\Gamma$ with generator $(\Delta_\Gamma, D(\Delta_\Gamma))$. Since $k<0$, the solution $u$ to the problem \eqref{E} has the probabilistic representation
\begin{align*}
u(t,x) 
= & \mathbf{E}_x\left[\mathbf{E}_{X^+_\tau}[u_0(X^{0, \Gamma} \circ L_t)]\right],  \quad t\geq 0,\; x \in \overline{\Omega}
\end{align*}
where $X^{0, \Gamma}$ on $\Gamma$ is generated by $(-A_0, D(-A_0))$.
\end{proof}

We now consider the case $\Lambda>0$.\begin{theorem}
\label{thm:timeChange}
Let $\Lambda>0$, $l \geq 0$, $k <0$ and $\alpha \in (0,1]$. The solution to \eqref{ELambda} has the probabilistic representation
\begin{align*}
u(t,x) 
= & \mathbf{E}_x\left[\mathbf{E}_{X^+_\tau}[u_0(X^\Gamma \circ L_t)]\right]\\
= & \mathbf{E}_x\left[e^{-\Lambda L_t} \mathbf{E}_{X^+_\tau}[u_0(X^{0, \Gamma} \circ L_t)]\right], \quad t\geq 0,\; x \in \overline{\Omega}.
\end{align*}
In particular,
\begin{align*}
u(t,x) 
= & \int_0^\infty e^{-\Lambda s} \mathbf{E}_x\left[\mathbf{E}_{X^+_\tau}[u_0(X^{0, \Gamma}_s)]\right]  \mathbf{P}_0(L_t \in ds)\\
= & \int_0^\infty e^{-\Lambda s} w(s,x) \mathbf{P}_0(L_t \in ds)
\end{align*}
where $w$ is the solution to \eqref{EsolVVw}.
\end{theorem}

\begin{proof}
The process $X^\Gamma$ on $\Gamma$ is generated by $(-A_\Lambda, D(-A_\Lambda))$ as defined above and $X^{0,\Gamma}$ on $\Gamma$ is right-continuous and generated by $(-A_0, D(-A_0))$ where $-A_0$ is given by $-A_{DN}$ and $\Delta_\Gamma$ without perturbation. Thus, $X^{0, \Gamma}$ is a Markov process with right-continuous paths for which the Feynmann-Kac formula
\begin{align*}
v(t,x) = \mathbf{E}_x \left[ e^{-\Lambda t} u_0(X^{0, \Lambda}_t) \right], \quad t >0,\, x \in \Gamma
\end{align*}
solves the problem to find $v \in C((0, \infty), \Gamma)$ such that
\begin{align}
\dot{v} = -A_0 v - \Lambda v, \quad v_0 = u_0 \in D(-A_0).
\label{CPproof}
\end{align}
That is, we solve $\dot{v} = -A_\Lambda v$ on $(0, \infty) \times \Gamma$. Thus, as in \eqref{eqProofHarm}, we proceed by considering
\begin{align}
\mathbf{E}_x[v(t,X^+_\tau)] = e^{-\Lambda t} \mathbf{E}_x\left[\mathbf{E}_{X^+_\tau}[u_0(X^{0, \Gamma}_t)]\right] = e^{-\Lambda t} w(t,x), \quad t\geq 0,\; x \in \overline{\Omega}
\label{TEMPw} 
\end{align}
in order to obtain a solution to \eqref{ELambda} with $\alpha=1$. The last equality is justified by the fact that $A_\Lambda$ is positive definite, then Theorem \ref{thm:SOLw} applies. The solution \eqref{TEMPw} coincides with $w$ as $\Lambda=0$.

For $\alpha \in (0,1)$, the fractional Cauchy problem associated with \eqref{CPproof} has a solution obtained via time-changed semigroup. According with the proof of Theorem \ref{th3} we can write
\begin{align*}
u(t,x) 
= & \mathbf{E}_x[v(L_t, X^+_\tau)]\\
= & \int_0^\infty e^{-\Lambda s} w(s,x) \mathbf{P}_0(L_t \in ds)\\
= &  \mathbf{E}_x\left[e^{-\Lambda L_t} \mathbf{E}_{X^+_\tau}[u_0(X^{0,\Gamma} \circ L_t)]\right]\\
= & \mathbf{E}_x\left[\mathbf{E}_{X^+_\tau}[u_0(X^\Gamma \circ L_t)]\right] \quad t\geq 0,\; x \in \overline{\Omega}.
\end{align*}
This concludes the proof.

\end{proof}

\section{General non local conditions }

The fractional problem \eqref{E} can be regarded as special case of the non-local problem
\begin{align}
\label{ENL}
\left\lbrace
\begin{array}{ll}
\displaystyle \Delta u=0 \qquad
&\text{in
$(0,\infty)\times\Omega$,}\\
 D_t^{\Phi} u = k \partial_{\bf n} u +l\Delta_\Gamma u - \Lambda u \qquad &\text{on
$(0,\infty)\times \Gamma$,}\\
u(0,x)=u_0(x) &
 \text{on $ \Gamma$.}
\end{array}
\right .
\end{align}

For the process $L$ with symbol $\Phi$ we introduce the function
\begin{align*}
M^\Lambda_\Phi(t, \lambda) : = \mathbf{E}_0[\exp - (\Lambda + \lambda) L_t], \quad (\lambda + \Lambda) \geq 0,\; \Lambda > 0,\; t>0.
\end{align*}

Also for this problem  we provide the compact representation of the solution to \eqref{ENL} which is strictly related with Theorem \ref{th2}.

\begin{theorem}
\label{th2NL}
Let $\Lambda=0$, $l > 0$, $k \in \mathbb{R}$ and $\alpha \in (0,1]$.  Let assume  $\lambda_1 + \Lambda \geq 0.$
Then for any $u_0 \in H^{s+2}( \Gamma)$ the solution $u$ of problem \eqref{ENL} is given by
\begin{align}
u(t, x) =\sum_{n=1}^\infty <u_0,\varphi_n> M^\Lambda_\Phi(t, \lambda_n)\psi_n(x), \quad t>0, \; x \in \overline{\Omega}
\label{uPhi}
\end{align}  
the series being convergent in $C((0,\infty) \times \bar\Omega)$.
\end{theorem}

\begin{proof}
First we observe that $\lambda_n + \Lambda \geq 0$ implies that the \eqref{uPhi} writes
\begin{align}
\label{eqPROOFuLambda}
u(t,x) = \sum_{n=1}^\infty <u_0,\varphi_n> \left( \int_0^\infty e^{-(\lambda_n + \Lambda) s} \mathbf{P}_0(L_t \in ds) \right) \psi_n(x), \quad t>0, \; x \in \overline{\Omega}
\end{align}
and, for a given $\Lambda>0$, $M^\Lambda_\Phi(t, \lambda_1) \leq M^\Lambda_\Phi(t, \lambda_n)$ $\forall\, n$. We can use the same arguments as in the proof of Theorem \ref{th2} together with the following facts:  as for $u_0 \in H^{s+2}(\Omega)$
\begin{align}
||| D^\Phi_t u(t,x) |||_{H^s(\Gamma)} \leq \sqrt{\sum_{n=1}^\infty |<u_0,\varphi_n>|^2 |\lambda_n + \Lambda|^{2+s} } < \infty 
\label{fact1}
\end{align}
and
\begin{align}
||| -A_\Lambda u |||_{H^s(\Gamma)} \leq ||| A_0 u + \Lambda u |||_{H^s(\Gamma)} \leq \sqrt{\sum_{n=1}^\infty |<u_0,\varphi_n>|^2 |\lambda_n + \Lambda|^{2+s} }< \infty 
\label{fact2}
\end{align}

By replacing \eqref{eqML} with
\begin{align}
D^\Phi_t M^\Lambda_\Phi(t, \lambda) = - (\lambda + \Lambda) M^\Lambda_\Phi(t, \lambda), \quad M^\Lambda_\Phi(0, \lambda) = 1
\label{eqM}
\end{align}
we can show that $u$ solves \eqref{ENL}. The equation \eqref{eqM} can be obtained from Theorem \ref{thm:eqM}.

Moreover, the series \eqref{uPhi} converges in $L^2(\overline{\Omega})$ uniformly in $(0, \infty)$ and, from the monotonicity of $M^\Lambda_\Phi (t,\lambda)$,
\begin{align*}
\sum_{n=1}^{\infty} || <u_0,\varphi_n> M^\Lambda_\Phi(t, \lambda_n) \varphi_n(x)||^2_{C([0,\infty), H^s(\Gamma))}< \infty
\end{align*}
implies that \eqref{uPhi} converges in $C((0, \infty)\times \overline{\Omega})$. \end{proof}

We point out that the probabilistic representation of \eqref{uPhi} can be obtained by following the same arguments as in Section 3.\\

We conclude by noting that this result can be useful to model delayed and rushed motions   through time change (see, for definitions and examples,   \cite{CapDovDelRus17}).

\vspace{1cm}

{\bf Acknowledgments}
The authors are supported by INdAM-GNAMPA and Sapienza University of Rome (Ateneo 2021 and Ateneo 2022). The authors thank MUR for the support under the project PRIN 2022 - 2022XZSAFN: Anomalous Phenomena on Regular and Irregular Domains: Approximating Complexity for the Applied Sciences - CUP B53D23009540006 (\url{https://www.sbai.uniroma1.it/~mirko.dovidio/prinSite/index.html}).

\end{document}